\documentclass[12pt]{amsart}

\usepackage[margin = 1in]{geometry}
\usepackage{amsfonts, amsmath, amscd, amssymb,
latexsym, mathrsfs, mathtools, 
slashed, stmaryrd, verbatim, wasysym, bbm, url}
\usepackage[utf8]{inputenc}
\usepackage[backend=bibtex,bibstyle=authoryear,citestyle=alphabetic,giveninits=true,doi=false,isbn=false,url=false,backref]{biblatex}
\usepackage[usenames]{xcolor}
\usepackage{colortbl}
\usepackage{enumerate}
\usepackage[shortlabels]{enumitem}
\usepackage{scalerel}
\usepackage{soul}
\usepackage{diagbox}
\usepackage{array}
\usepackage{nicematrix}
      
\DeclareFieldFormat{labelalphawidth}{\mkbibbrackets{#1}}

\defbibenvironment{bibliography}
  {\list
     {\printtext[labelalphawidth]{%
        \printfield{labelprefix}%
    \printfield{labelalpha}%
        \printfield{extraalpha}}}
     {\setlength{\labelwidth}{\labelalphawidth}%
      \setlength{\leftmargin}{\labelwidth}%
      \setlength{\labelsep}{\biblabelsep}%
      \addtolength{\leftmargin}{\labelsep}%
      \setlength{\itemsep}{\bibitemsep}%
      \setlength{\parsep}{\bibparsep}}%
      }
  {\endlist}
  {\item}
\DeclareNameAlias{author}{last-first}

\newtheorem*{acknowledgements}{Acknowledgements}
\newtheorem*{theorem*}{Theorem}
\newtheorem{theorem}{Theorem}
\newtheorem{corollary}{Corollary}

\newtheorem{proposition}[corollary]{Proposition}
\theoremstyle{definition}

\newtheorem{example}{Example}
\newtheorem{remark}[example]{Remark}

\numberwithin{equation}{section}
\let\oldsqrt\sqrt
\def\sqrt{\mathpalette\DHLhksqrt}
\def\DHLhksqrt#1#2{%
\setbox0=\hbox{$#1\oldsqrt{#2\,}$}\dimen0=\ht0
\advance\dimen0-0.2\ht0
\setbox2=\hbox{\vrule height\ht0 depth -\dimen0}%
{\box0\lower0.4pt\box2}}

\usepackage{verbatim} %% code from mathabx.sty and mathabx.dcl
\DeclareFontFamily{U}{mathx}{\hyphenchar\font45}
\DeclareFontShape{U}{mathx}{m}{n}{
      <5> <6> <7> <8> <9> <10>
      <10.95> <12> <14.4> <17.28> <20.74> <24.88>
      mathx10
      }{}
\DeclareSymbolFont{mathx}{U}{mathx}{m}{n}
\DeclareFontSubstitution{U}{mathx}{m}{n}
\DeclareMathAccent{\widecheck}{0}{mathx}{"71}

\renewcommand{\tilde}[1]{\widetilde{#1}}
\renewcommand{\bar}{\overline}

\newcommand\eps\varepsilon
\renewcommand\epsilon\varepsilon

\newcommand{\abs}[1]{\left\lvert #1 \right\rvert}

\newcommand\inner[1]{\langle #1 \rangle}

\newcommand\paperintro%
        {%
         }
\newcommand\paperbody%
        {%
         }

\newcommand\bbG{\mathbb{G}}
\newcommand\bbH{\mathbb{H}}
\newcommand\bbI{\mathbb{I}}

\newcommand\bbN{\mathbb{N}}

\newcommand\bbR{\mathbb{R}}
\newcommand\bbS{\mathbb{S}}

\newcommand\cF{\mathcal{F}}

\newcommand\cN{\mathcal{N}}

\newcommand\cW{\mathcal{W}}

\newcommand\mf[1]{\mathfrak{ #1}}

\DeclareMathAlphabet{\mathpzc}{OT1}{pzc}{m}{it}

\newcommand{\sbs}{\subset}

\newcommand{\csob}{C^{\mathrm{Sob}}}
\newcommand{\cfk}{C^{\mathrm{FK}}}
\newcommand{\tgamma}[1]{ \tilde{\gamma}_{#1} }

\usepackage[refpage]{nomencl}

\makenomenclature

\usepackage{tikz, tikz-cd, tkz-euclide}
\usetikzlibrary{calc,positioning,intersections,quotes,decorations.markings}
\makeatletter
\def\@tocline#1#2#3#4#5#6#7{\relax
  \ifnum #1>\c@tocdepth % then omitth
  \else
    \par \addpenalty\@secpenalty\addvspace{#2}%
    \begingroup \hyphenpenalty\@M
    \@ifempty{#4}{%
      \@tempdima\csname r@tocindent\number#1\endcsname\relax
    }{%
      \@tempdima#4\relax
    }%
    \parindent\z@ \leftskip#3\relax \advance\leftskip\@tempdima\relax
    \rightskip\@pnumwidth plus4em \parfillskip-\@pnumwidth
    #5\leavevmode\hskip-\@tempdima
      \ifcase #1
       \or\or \hskip 1em \or \hskip 2em \else \hskip 3em \fi%
      #6\nobreak\relax
    \hfill\hbox to\@pnumwidth{\@tocpagenum{#7}}\par% <---- \dotfill -> \hfill
    \nobreak
    \endgroup
  \fi}
\makeatother

\def\annu#1{_{% 
  \vbox{\hrule height .2pt 
    \kern 1pt 
    \hbox{$\scriptstyle {#1}\kern 1pt$}% 
  }\kern-.05pt 
  \vrule width .2pt 
}}

\allowdisplaybreaks
\usepackage[colorlinks,citecolor=blue,urlcolor=red,bookmarks=false,hypertexnames=true]{hyperref} 

\makeatletter
\def\keywords{\xdef\@thefnmark{}\@footnotetext}
\makeatother

\title{A note on the Pleijel theorem for $H$-type groups}
\author{Yaozhong Qiu}
\address{UPL, Univ. Paris Nanterre, CNRS, F92000 Nanterre France}
\email{yqiu@parisnanterre.fr}

\begin{document}
\begin{abstract}
We continue the program initiated by \cite{frank2024courant} and show that the Pleijel theorem holds unconditionally on all but four $H$-type groups. 
\end{abstract}


\maketitle

\section{Introduction and main result}
In this note, we continue the study of the Pleijel theorem for subriemannian laplacians defined on the product of a Heisenberg group $\bbH_n$ and $\bbR^k$ initiated by Frank and Helffer \cite{frank2024courant}, and extend their argument to all but four $H$-type groups. We will follow closely their proof which relates the asymptotic behaviour of the number of nodal domains to the sharp constant of the $L^2$-Sobolev inequality and the Weyl constant appearing in the asymptotics of the eigenvalue counting function. 

Let $\bbG$ be a step two stratified Lie group with Lie algebra $\mf{g}$, meaning $\mf{g}$ admits the stratification $\mf{g} = V_1 \oplus [V_1, V_1]$ for some linear subspace $V_1 \sbs \mf{g}$. We say $\bbG$ is a $H$-type group if $\mf{g}$ is equipped with an inner product $\inner{\cdot, \cdot}$ with the property if $\mf{z}$ denotes the centre of $\mf{g}$, then 
\begin{enumerate}
    \item $[\mf{g}^\perp, \mf{g}^\perp] = \mf{g}$ and
    \item for every fixed $z \in \mf{g}$, the map $J_z: \mf{z}^\perp \mapsto \mf{z}^\perp$ defined by 
\[ \inner{J_z(v), w} = \inner{z, [v, w]} \] for all $w \in \mf{z}^\perp$ is orthogonal for all $\inner{z, z} = 1$. 
\end{enumerate}
If $n = \dim(\mf{z}^\perp)$ and $m = \dim(\mf{z})$, in which case the stratification is $\mf{g} = \mf{z}^\perp \oplus \mf{z}$, then $\bbG$ can be equipped with a system of coordinates $(x, t) \in \bbR^n_x \times \bbR^m_t$ equipped with the group law
\[ (x, t) \circ (\xi, \tau) = \left(x + \xi, t_1 + \tau_1 + \frac{1}{2}\inner{U^{(1)}x, \xi}, \cdots, t_n + \tau_n + \frac{1}{2}\inner{U^{(n)}x, \xi}\right) \]
for a collection $U^{(1)}, \cdots, U^{(m)}$ of $n \times n$ skew-symmetric, orthogonal, and anticommuting matrices, see for instance \cite[Theorem~18.2.1]{bonfiglioli2007stratified}. These conditions imply the $U^{(1)}, \cdots, U^{(m)}$ are linearly independent and that $n$ is even. Consequently in the sequel we shall always write $2n = \dim(\mf{z}^\perp)$ for some $n \in \bbN$ and $\bbG \cong \bbR^{2n}_x \times \bbR^m_t$. Associated with the group law is a canonical family of vector fields $X_1, \cdots, X_{2n}$. The canonical negative sublaplacian on $\bbG$ is given by
\begin{equation}\label{htype-sublaplacian}
    \Delta^\bbG = -\sum_{i=1}^{2n} X_i^*X_i = \Delta^{\bbR^{2n}}_x + \frac{1}{4}\abs{x}^2\Delta^{\bbR^m}_t + \sum_{j=1}^m \inner{U^{(j)}x, \nabla_x}\partial_{t_j} 
\end{equation}
where $X_i^*$ is the adjoint of $X_i$ with respect to Lebesgue measure $d\xi$ on $\bbG$. 

We recall the results in the setting of \cite{frank2024courant}. Let $\bbH_n$ be the $(2n+1)$-dimensional Heisenberg group, a $H$-type group according to the previous definition with $m = 1$ and $U^{(1)}$ a block diagonal matrix of the form
\[ U^{(1)} = \left(\begin{array}{cc} 0 & -\bbI_{n} \\ \bbI_{n} & 0 \\ \end{array}\right) \]
Consider the space $\bbH_n \times \bbR^k$, where $n \in \bbN$ and $k \in \bbN_0$, equipped with its sublaplacian 
\[ \Delta^{\bbH_n \times \bbR^k} = \Delta^{\bbH^n} \otimes \bbI^{\bbR^k} + \bbI^{\bbH_n} \otimes \Delta^{\bbR^k}. \] 
If $\Omega \sbs \bbH_n \times \bbR^k$ is a domain of finite measure, it was shown that $\smash{-\Delta_\Omega^{\bbH_n \times \bbR^k}}$, the selfadjoint Dirichlet realisation of $\smash{-\Delta^{\bbH_n \times \bbR^k}}$ on $\Omega$, has discrete spectrum and we denote by $(\lambda_\ell(\Omega))_{\ell \geq 1}$ its eigenvalues, which are arranged in nondecreasing order and counted with multiplicity, and by $(\varphi_\ell(\Omega))_{\ell \geq 1}$ its corresponding eigenfunctions.

If $f$ is continuous on $\Omega$, then its nodal set is 
\[ \cN(f) = \overline{\{x \in \Omega \mid f(x) = 0\}} \] and a nodal domain is a connected component of $\Omega \setminus \cN(f)$. Since it was also shown the eigenfunctions $\varphi_\ell(\Omega)$ are smooth in $\Omega$, the nodal domains are well defined and we denote by $\nu_\ell(\Omega)$ the number of nodal domains of $\varphi_\ell(\Omega)$. Then \cite[Theorems~7.1~and~9.1]{frank2024courant} assert
\begin{align}
   \limsup_{\ell \to \infty} \frac{\nu_\ell(\Omega)}{\ell} \leq \gamma(\bbH_n \times \bbR^k) \notag &= \cfk(\bbH_n \times \bbR^k)^{-Q/2}\cW(\bbH_n \times \bbR^k)^{-1} \\
   &\leq \csob(\bbH_n \times \bbR^k)^{-Q/2}\cW(\bbH_n \times \bbR^k)^{-1}, \label{gamma-bound}
\end{align}
where $\cfk(\bbH_n \times \bbR^k)$ and $\cW(\bbH_n \times \bbR^k)$ are the Faber-Krahn and Weyl constants respectively appearing in the lower bound on the first eigenvalue 
\begin{equation}\label{faber-krahn}
    \lambda_1(\Omega) \geq \cfk(\bbH_n \times \bbR^k) \abs{\Omega}^{-2/Q} 
\end{equation}
and the Weyl asymptotics for the eigenvalue counting function 
\begin{equation}\label{weyl-asymptotics}
    \lim_{\lambda \to \infty} N(\lambda, -\Delta_\Omega^{\bbH_n \times \bbR^k}) = \abs{\{\ell \in \bbN \mid \lambda_\ell(\Omega) < \lambda\}} \sim \cW(\bbH_n \times \bbR^k)\abs{\Omega} \lambda^{Q/2} 
\end{equation}
and where $Q = Q(n, k) = 2n + 2 + k$ is the homogeneous dimension of $\bbH_n \times \bbR^k$, and $\csob(\bbH_n \times \bbR^k)$ is the sharp constant in the $L^2$-Sobolev inequality 
\begin{equation}\label{sobolev-inequality}
    \int_{\bbH_n \times \bbR^k} (-\Delta^{\bbH_n \times \bbR^k}u)\bar{u}d\xi \geq \csob(\bbH_n \times \bbR^k) \left(\int_{\bbH_n \times \bbR^k} \abs{u}^{2Q/(Q - 2)}d\xi\right)^{(Q-2)/Q}
\end{equation}
valid for $u \in C_0^\infty(\bbH_n \times \bbR^k)$. 

On $\bbR^k$, the Courant theorem \cite{courant1923allgemeiner} asserts $\nu_\ell(\Omega) \leq \ell$, that is $\varphi_\ell$ has at most $\ell$ nodal domains, while the Pleijel theorem \cite[\S5]{pleijel1956remarks} asserts the existence of a constant $\gamma(\bbR^k)$ independent of $\Omega$ such that 
\[ \limsup_{\ell \to \infty} \frac{\nu_\ell(\Omega)}{\ell} \leq \gamma(\bbR^k) \]
and $\gamma(\bbR^k) < 1$ for $k \geq 2$, that is $\nu_\ell(\Omega) = \ell$ for finitely many $\ell$. 

It was shown in \cite[Theorem~7.2]{frank2024courant} that Pleijel's theorem $\gamma(\bbH_n \times \bbR^k) < 1$ holds unconditionally for all but four pairs of $(n, k) \in \bbN \times \bbN_0$ as a consequence of \eqref{gamma-bound} and the sharp constant in the $L^2$-Sobolev inequality for Heisenberg groups \cite[Corollary~C]{jerison1988extremals}, see also \cite[Theorem~2.1]{frank2012sharp}, and otherwise holds for all pairs of $(n, k)$ assuming the validity of the Pansu conjecture \cite{pansu1983isoperimetric} concerning the isoperimetric problem on the Heisenberg group and which gives a better bound than \eqref{gamma-bound} on $\cfk(\bbH_n \times \bbR^k)$, see \cite[Proposition~11.1]{frank2024courant}. The goal of this note is to modestly extend the former result to $H$-type groups using the recent result of Yang for the sharp constant in the $L^2$-Sobolev inequality \cite{yang2024optimal}. 

\begin{theorem}
    Let $\bbG \cong \bbR^{2n}_x \times \bbR^m_t$ be a $H$-type group. Then Pleijel's theorem $\gamma(\bbG) < 1$ holds for all but $(n, m) \in \{(1, 1), (2, 1), (3, 1), (2, 2)\}$.  
\end{theorem}

\begin{remark}
    If $m = 1$ then $\bbG$ is isomorphic to a Heisenberg group. Otherwise, see \cite[Example~18.1.3]{bonfiglioli2007stratified} for an example of a $H$-type group with $(n, m) = (2, 2)$. 
\end{remark}
\section{Proof of main result}

We consider the analogue of \eqref{gamma-bound} on a $H$-type group which reads 
\[ \limsup_{\ell \to \infty} \frac{\nu_\ell(\Omega)}{\ell} \leq \csob(\bbG)^{-Q/2}\cW(\bbG)^{-1} =\vcentcolon \tilde{\gamma}(\bbG) \] 
for $Q = Q(n, m) = 2n + 2m$ the homogeneous dimension of $\bbG$. The technical details, for instance concerning the well-definedness of the nodal domain count $\nu_\ell(\Omega)$, the existence of the Weyl constant $\cW(\bbG)$, or the validity of \eqref{gamma-bound}, follow mutatis mutandis from the discussion given in \cite{frank2024courant}. For brevity and since the results depend only on the dimensions $2n$ and $m$, in the sequel we shall always indicate this dependence in the subscript.

First, by \cite[Theorem~1.2]{yang2024optimal}, we have 
\begin{equation}\label{Htype-sobolev-inequality}
    \csob_{n, m} = 4^{n/(n + m)}n(n + m - 1)\pi^{(2n + m)/(2n + 2m)}\left(\frac{\Gamma(n + m/2)}{\Gamma(2n + m)}\right)^{1/(n + m)}.
\end{equation}
Note a $H$-type group in \cite{yang2024optimal} was defined with $m = \dim(\mf{z}^\perp)$ and $n = \dim(\mf{z})$, so \eqref{Htype-sobolev-inequality} comes from replacing $m$ and $n$ with $2n$ and $m$ respectively in their Sobolev constant. Second, to compute the Weyl constant, we follow the diagonalisation procedure in \cite{hansson2008sharp, frank2024courant}.

\begin{proposition}
    Let $\bbG \cong \bbR^{2n}_x \times \bbR^m_t$ be a $H$-type group. Let 
    \begin{equation}\label{laptevC}
        c_{n, m} = \sum_{k \geq 0} \binom{k + n - 1}{k} \frac{1}{(2k + n)^{n+m}}.
    \end{equation}
    Then 
    \begin{equation}\label{htype-weyl-asymptotics}
        \cW_{n, m} = \frac{\omega_{m-1}}{(2\pi)^{n+m}}\frac{1}{n+m}c_{n, m}
    \end{equation}
    where $\omega_{m-1}$ is the surface volume of the $(m-1)$-dimensional sphere $\bbS^{m-1}$. In particular, 
    \begin{equation}\label{htype-gamma-bound}
        \limsup_{\ell \to \infty} \frac{\nu_\ell(\Omega)}{\ell} \leq \tgamma{n, m} \vcentcolon= \frac{1}{2^{n-m+1}}\frac{n+m}{n^{n+m}(n+m-1)^{n+m}} \frac{\Gamma(m/2)\Gamma(2n+m)}{\Gamma(n+m/2)}\frac{1}{c_{n, m}}. 
    \end{equation}
\end{proposition}

\begin{proof}
    Reasoning as in \cite[Equation~8.2]{frank2024courant}, we want to show 
    \[ \mathbbm{1}(-\Delta^{\bbG} < \lambda)((x, t), (x, t)) = \cW_{n, m}\lambda^{n+m} \]
    for some $\cW_{n, m} > 0$. By taking the partial Fourier transform in the central variables $t \in \bbR^m_t$ with $\tau$ the Fourier variable dual to $t$, we see that while $\cF_t(-\Delta^\bbG) \cF_t^* \vcentcolon = -\Delta_\tau^\bbG$ is \emph{not} a sum of $n$ decoupled Landau hamiltonians  
    \[ -\Delta_{\abs{\tau}}^{\bbH_n} = -\sum_{j=1}^n \left(\partial_{x_j} + \frac{i}{2}x_{j+n}\abs{\tau}\right)^2 + \left(\partial_{x_{j+n}} - \frac{i}{2}x_j\abs{\tau}\right)^2 \]
    for $\tau \neq 0$, as in the case of the Heisenberg group, they \emph{are} unitarily equivalent, that is there exists an orthogonal matrix $T_\tau \in O(\bbR^{2n}_x)$ such that $\smash{-\Delta_\tau^\bbG = T_\tau \circ (-\Delta_{\abs{\tau}}^{\bbH_n}) \circ T_\tau^{-1}}$, see for instance \cite[Lemma~2.1]{niedorf2024p}. Consequently their spectra agree, and since the densities of the eigenprojections of a single Landau hamiltonian are constant on the diagonal, the densities of the $T_\tau$-conjugated eigenprojections are equal on the diagonal to that same constant and hence each $\tau$-fibre once again contributes $(2\pi)^{-n}\abs{\tau}^n$ to the on-diagonal spectral density. Then taking the inverse partial Fourier transform, we obtain
    \begin{align*}
        \mathbbm{1}(-\Delta^\bbG < \lambda)((x, t), (x, t)) &= \int_{\bbR^m} \frac{d\tau}{(2\pi)^m} \frac{\abs{\tau}^n}{(2\pi)^n} \sum_{k \in \bbN_0^n} \mathbbm{1}(\abs{\tau}(2\abs{k} + n) < \lambda) \\ 
        &=  \frac{\omega_{m-1}}{(2\pi)^{n+m}} \int_0^\infty \tau^{n+m-1} \sum_{k \in \bbN_0^n} \mathbbm{1}(\tau(2\abs{k} + n) < \lambda) d\tau \\
        &= \frac{\omega_{m-1}}{(2\pi)^{n+m}}\frac{1}{n+m} \sum_{k \in \bbN_0^n} \left(\frac{\lambda}{2(\abs{k} + n)}\right)^{n+m} \\
        &= \frac{\omega_{m-1}}{(2\pi)^{n+m}}\frac{\lambda^{n+m}}{n+m} \sum_{k \geq 0} \binom{k+n-1}{k}\frac{1}{(2k + n)^{n+m}}. 
    \end{align*}
    Note that if $m = 1$ then \eqref{htype-weyl-asymptotics} recovers \cite[Equation~8.3]{frank2024courant} up to a factor of $4$, which is due to the choice of normalisation used in their definition of the Heisenberg group, and \cite[Equation~9.2]{frank2024courant}, which is independent of the normalisation.  
\end{proof} 

We now show $\tgamma{n, m} < 1$ for all but finitely many pairs of $(n, m) \in \bbN \times \bbN$. In fact it suffices to consider a strictly smaller subset since it turns out that there are constraints on $n$ and $m$, characterised by \cite[Corollary~1]{kaplan1980fundamental}, in the sense there are no $H$-type groups for some $(n, m)$. In particular, if $n \in \{1, 3\}$ then necessarily $m = 1$, and if $n = 2$ then $m \leq 3$. These are precisely the four exceptions given in the statement of the theorem, excluding the specific case of $(n, m) = (2, 3)$ which we treat separately. Note if $n$ is odd then the centre is always $1$-dimensional and by \cite[Remark~18.2.6]{bonfiglioli2007stratified} it follows $\bbG \cong \bbH_n$. 

We first prove $\tgamma{n, m}$ decreases with $n \geq 1$ for all $m \geq 1$. Consider the quotient for $n \geq 2$
\begin{equation}\label{psi-ratio}
    \varphi_{n, m} \vcentcolon= \frac{\tgamma{n, m}}{\tgamma{n-1, m}} = \frac{(n-1)^{n+m-1}(n+m-2)^{n+m-1}(n+m)(2n+m-1)}{n^{n+m}(n+m-1)^{n+m+1}} \frac{c_{n-1, m}}{c_{n, m}}. 
\end{equation}
In order to prove a lower bound for $c_{n, m}/c_{n-1, m}$, we study the quotient
\begin{equation}\label{kth-term-quotient}
    \frac{\binom{k+n-1}{k}(2k+n)^{-n-m}}{\binom{k+n-2}{k}(2k+n-1)^{-n-m+1}} = \frac{1}{n-1}\frac{k+n-1}{2k+n}\left(1 - \frac{1}{2k+n}\right)^{n+m-1}
\end{equation}
of the $k$-th term in the numerator to the $k$-th term in the denominator. The derivative with respect to $k$ of the right hand side is $1/(n-1)$ multiplied by 
\[ \frac{d}{dk}\left(\frac{k+n-1}{2k+n}\left(1 - \frac{1}{2k+n}\right)^{n+m-1}\right) = \frac{2k(m+1)+(n-1)(n+2m)}{(2k+n-1)^2(2k+n)}\left(1 - \frac{1}{2k+n}\right)^{n+m} \]
which is positive. The right hand side of \eqref{kth-term-quotient} therefore minimises at $k = 0$ and hence
\[ \frac{c_{n, m}}{c_{n-1, m}} \geq \frac{1}{n}\left(1 - \frac{1}{n}\right)^{n+m-1}. \] 
Inserting this into \eqref{psi-ratio} reduces the problem to showing
\[ \varphi_{n, m} \leq \frac{(n+m-2)^{n+m-1}(n+m)(2n+m-1)}{(n+m-1)^{n+m+1}} \leq \frac{1}{e} \frac{(n+m)(2n+m-1)}{(n+m-1)^2} < 1. \]
It suffices to bound the fraction from above by $\frac{5}{2}$ since
\begin{align*}
    (n+m)(2n + m-1) - \tfrac{5}{2}(n + m -1)^2 &= -\tfrac{3}{2}m^2 - 2(n-2)m - \tfrac{1}{2}(n^2 - 8n + 5) \\
    &\leq -\tfrac{3}{2}m^2 - 4m + \tfrac{11}{2} 
\end{align*}
is nonpositive for $m \geq 1$ so $\varphi_{n, m} \leq \frac{5}{2e} < 1$ and hence $\tgamma{n, m}$ is monotonic in $n$ for each $m$. 

On the other hand, although the numerics (see appendix) appear to support the hypothesis $\tgamma{n, m}$ also decreases with respect to $m$ uniform in $n$, it is not clear how to generalise the previous proof since the analogue of \eqref{kth-term-quotient}, that is the quotient of the $k$-th term in $c_{n, m}$ and $c_{n, m-1}$ respectively, is $1/(2k+n)$ which is not bounded below. 

However, if we first bound $c_{n, m}$ from below by truncating at the first term $n^{-(n+m)}$ in the series so that  
\[ \tgamma{n, m} \leq \bar{\gamma}_{n, m} \vcentcolon= \frac{1}{2^{n-m+1}}\frac{n+m}{(n+m-1)^{n+m}}\frac{\Gamma(m/2)\Gamma(2n+m)}{\Gamma(n+m/2)}, \]
we can prove $\bar{\gamma}_{n, m}$ decreases with $m \geq 1$ for all $n \geq 1$. Consider the quotients for $m \geq 2$ 
\begin{align}
    \psi_{n, m} \vcentcolon= \frac{\bar{\gamma}_{n, m}}{\bar{\gamma}_{n, m-1}} &= 4\frac{(n+m-2)^{n+m-1}(n+m)}{(n+m-1)^{n+m+1}} \frac{\Gamma(m/2)\Gamma(n+m/2+1/2)}{\Gamma(m/2-1/2)\Gamma(n+m/2)} \notag \\
    &\leq \frac{4}{e} \frac{n+m}{(n+m-1)^2}\frac{\Gamma(m/2)\Gamma(n+m/2+1/2)}{\Gamma(m/2-1/2)\Gamma(n+m/2)}. \label{psi-ratio}
\end{align}
We distinguish the particular case $n = 1$ where \eqref{psi-ratio} reads 
\[ \psi_{1, m} \leq \frac{4}{e}\frac{m+1}{m^2}\frac{(m-1)(m+1)}{2m} = \frac{2}{e}\frac{(m-1)(m+1)^2}{m^3}. \]
The derivative of the right hand side is zero at $m = 3$ where the second derivative is negative so $\psi_{1, m} \leq \frac{64}{27e} < 1$. Otherwise for $n \geq 2$ and replacing $m$ with $m + 1$ for convenience, by Wendel's inequality
\[ \psi_{n, m+1} \leq \frac{4}{e}\frac{n+m+1}{(n+m)^2}\sqrt{m/2}\sqrt{n+m/2+1/2}. \]
Squaring and performing a change of variable $\ell = n + m \geq 2 + m \geq 3$ we find 
\begin{align*}
    \psi_{n, m+1}^2 \leq \frac{4}{e^2} \frac{m(n+m+1)^2(2n+m+1)}{(n+m)^4} &= \frac{4}{e^2}\frac{m(\ell + 1)^2(2\ell - m + 1)}{\ell^4} \\
    &= \frac{4}{e^2}\frac{(\ell+1)^2}{\ell^4}(-m^2+(2\ell+1)m).
\end{align*}
For each fixed $\ell$ the parabola $-m^2 + (2\ell + 1)m$ is increasing on $m \in (0, \ell + \frac{1}{2})$. But since $m \leq \ell - 2$, it maximises at this boundary and we obtain
\[ \psi_{n, m+1}^2 \leq \frac{4}{e^2}\frac{(\ell+1)^2(\ell^2+\ell-6)}{\ell^4} \leq \frac{4}{e^2}\left(1 + \frac{3}{\ell} - \frac{3}{\ell^2}\right). \]
The derivative of the right hand side is negative for $\ell \geq 3$ so $\psi_{n, m+1}^2 \leq \frac{20}{3e^2} < 1$ and hence $\bar{\gamma}_{n, m}$ is monotonic in $m$ for each $n$. 

Combining the two monotonicity arguments yields for $n \geq 4$ and $m \geq 2$:
\[ \tgamma{n, m} \leq \tgamma{4, m} \leq \overline{\gamma}_{4, m} \leq \overline{\gamma}_{4, 2} = \frac{2268}{3125} \approx 0.7258 < 1 \]
while the case $m = 1$ is the content of \cite[Theorem~7.2]{frank2024courant}. All that remains is to also verify $\tgamma{2, 3} < 1$ and indeed $\tgamma{2, 3} \leq \overline{\gamma}_{2, 3} = \frac{15}{16} = 0.9375 < 1$. 

In light of the previous results of Frank and Helffer for $\bbH_n \times \bbR^k$ and the present results for $H$-type groups $\bbG_{n, m} \cong \bbR^{2n}_x \times \bbR^m_t$, it seems reasonable to expect the Pleijel theorem (at least when deduced via the $L^2$-Sobolev constant, not the Faber-Krahn constant) holds also for $\bbG_{n, m} \times \bbR^k$ with the exceptions depending only on $n$, $m$, and $k$, or better yet perhaps only on their (homogeneous) sum. Another possible direction may be extension to the larger class of M\'etivier groups where the sharp $L^2$-Sobolev constant is not known and the spectral decomposition of the sublaplacian is more complicated, see for instance \cite[\S2]{martini2014spectral}. 

\begin{acknowledgements}
\textup{We are grateful to Rupert Frank and Bernard Helffer for valuable discussions which helped improve the monotonicity argument and greatly reduced the number of cases needing manual verification. This project has received funding from the European Union’s Horizon 2020 research and innovation programme under the Marie Skłodowska-Curie grant agreement No 101034255. \scalerel*{\includegraphics{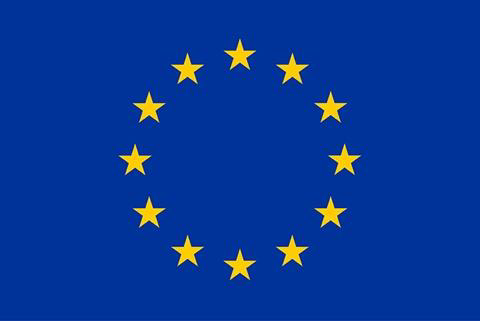}}{A}}
\end{acknowledgements}

\section*{Appendix}
For the reader's convenience we collect some numeric computations of $\tgamma{n, m}$ and $\overline{\gamma}_{n, m}$. Rows correspond to fixed $n$, and columns to fixed $m$. Grey shading indicates inadmissible pairs of $(n, m)$ while red highlighting indicates values greater than $1$. Note in the first figure the first column agrees with the numerics found in \cite[p.~44]{frank2024courant}. 

\small
\begin{figure}
    \centering
    \begin{NiceTabular}{|c|c|c|c|c|c|c|c|c|c|c|}
    \CodeBefore
    \cellcolor{gray!50}{
    2-3,2-4,2-5,2-6,2-7,2-8,2-9,2-10,2-11,
    3-5,3-6,3-7,3-8,3-9,3-10,3-11,
    4-3,4-4,4-5,4-6,4-7,4-8,4-9,4-10,4-11,
    5-9,5-10,5-11,
    6-3,6-4,6-5,6-6,6-7,6-8,6-9,6-10,6-11,
    7-5,7-6,7-7,7-8,7-9,7-10,7-11,
    8-3,8-4,8-5,8-6,8-7,8-8,8-9,8-10,8-11,
    9-10,9-11,
    10-3,10-4,10-5,10-6,10-7,10-8,10-9,10-10,10-11,
    11-5,11-6,11-7,11-8,11-9,11-10,11-11}
    \Body
    \hline
        $n/m$ & 1 & 2 & 3 & 4 & 5 & 6 & 7 & 8 & 9 & 10 \\ \hline
        1 & \textcolor{red}{3.2423} & 2.1392 & 1.5574 & 1.1666 & 0.8835 & 0.6718 & 0.5115 & 0.3893 & 0.2960 & 0.2248 \\ \hline
        2 & \textcolor{red}{1.8238} & \textcolor{red}{1.2325} & 0.8662 & 0.6221 & 0.4530 & 0.3329 & 0.2462 & 0.1828 & 0.1361 & 0.1015 \\ \hline
        3 & \textcolor{red}{1.0689} & 0.7141 & 0.4892 & 0.3413 & 0.2414 & 0.1726 & 0.1244 & 0.0903 & 0.0659 & 0.0482 \\ \hline 
        4 & 0.6249 & 0.4120 & 0.2771 & 0.1893 & 0.1310 & 0.0917 & 0.0647 & 0.0461 & 0.0330 & 0.0237 \\ \hline
        5 & 0.3626 & 0.2365 & 0.1568 & 0.1054 & 0.0718 & 0.0494 & 0.0343 & 0.0240 & 0.0169 & 0.0120 \\ \hline 
        6 & 0.2089 & 0.1350 & 0.0885 & 0.0588 & 0.0395 & 0.0268 & 0.0184 & 0.0127 & 0.0088 & 0.0062 \\ \hline 
        7 & 0.1196 & 0.0767 & 0.0499 & 0.0328 & 0.0218 & 0.0146 & 0.0099 & 0.0068 & 0.0046 & 0.0032 \\ \hline
        8 & 0.0681 & 0.0434 & 0.0280 & 0.0183 & 0.0120 & 0.0080 & 0.0054 & 0.0036 & 0.0025 & 0.0017 \\ \hline
        9 & 0.0386 & 0.0245 & 0.0157 & 0.0102 & 0.0067 & 0.0044 & 0.0029 & 0.0020 & 0.0013 & 0.0010 \\ \hline
        10 & 0.0218 & 0.0138 & 0.0088 & 0.0057 & 0.0037 & 0.0024 & 0.0016 & 0.0011 & 0.0007 & 0.0005 \\ \hline
    \end{NiceTabular}
    \caption{Values of $\tilde{\gamma}_{n, m}$ for $1 \leq n, m \leq 10$.}
\end{figure}
\normalsize

\small
\begin{figure}
    \centering
    \begin{NiceTabular}{|c|c|c|c|c|c|c|c|c|c|c|}
    \CodeBefore
    \cellcolor{gray!50}{
    2-3,2-4,2-5,2-6,2-7,2-8,2-9,2-10,2-11,
    3-5,3-6,3-7,3-8,3-9,3-10,3-11,
    4-3,4-4,4-5,4-6,4-7,4-8,4-9,4-10,4-11,
    5-9,5-10,5-11,
    6-3,6-4,6-5,6-6,6-7,6-8,6-9,6-10,6-11,
    7-5,7-6,7-7,7-8,7-9,7-10,7-11,
    8-3,8-4,8-5,8-6,8-7,8-8,8-9,8-10,8-11,
    9-10,9-11,
    10-3,10-4,10-5,10-6,10-7,10-8,10-9,10-10,10-11,
    11-5,11-6,11-7,11-8,11-9,11-10,11-11}
    \Body
    \hline
        $n/m$ & 1 & 2 & 3 & 4 & 5 & 6 & 7 & 8 & 9 & 10 \\ \hline
        1 & \textcolor{red}{4.0000} & 2.2500 & 1.5803 & 1.1719 & 0.8847 & 0.6722 & 0.5116 & 0.3893 & 0.2960 & 0.2248 \\ \hline
        2 & \textcolor{red}{3.0000} & \textcolor{red}{1.4815} & 0.9375 & 0.6451 & 0.4609 & 0.3357 & 0.2472 & 0.1832 & 0.1362 & 0.1016 \\ \hline
        3 & \textcolor{red}{2.3704} & 1.0254 & 0.5898 & 0.3781 & 0.2558 & 0.1784 & 0.1269 & 0.0913 & 0.0663 & 0.0484 \\ \hline
        4 & \textcolor{red}{1.8750} & 0.7258 & 0.3841 & 0.2308 & 0.1483 & 0.0992 & 0.0681 & 0.0476 & 0.0337 & 0.0241 \\ \hline
        5 & \textcolor{red}{1.4746} & 0.5199 & 0.2558 & 0.1450 & 0.0888 & 0.0571 & 0.0379 & 0.0257 & 0.0178 & 0.0124 \\ \hline
        6 & \textcolor{red}{1.1523} & 0.3751 & 0.1730 & 0.0930 & 0.0545 & 0.0337 & 0.0217 & 0.0143 & 0.0096 & 0.0066 \\ \hline
        7 & 0.8953 & 0.2718 & 0.1184 & 0.0606 & 0.0341 & 0.0204 & 0.0127 & 0.0081 & 0.0054 & 0.0036 \\ \hline
        8 & 0.6921 & 0.1977 & 0.0817 & 0.0401 & 0.0217 & 0.0125 & 0.0076 & 0.0047 & 0.0030 & 0.0020 \\ \hline
        9 & 0.5329 & 0.1440 & 0.0568 & 0.0267 & 0.0140 & 0.0078 & 0.0046 & 0.0028 & 0.0018 & 0.0011 \\ \hline
        10 & 0.4087 & 0.1051 & 0.0397 & 0.0180 & 0.0091 & 0.0049 & 0.0028 & 0.0017 & 0.0010 & 0.0006 \\ \hline
    \end{NiceTabular}
    \caption{Values of $\overline{\gamma}_{n, m}$ for $1 \leq n, m \leq 10$.}
\end{figure}
\normalsize

\printbibliography

\end{document}